\documentclass{article}
\usepackage{graphicx}
\usepackage{subfigure}
\usepackage{algorithm}
\usepackage{float}
\usepackage{amsmath,amssymb,amsthm,enumerate,mathrsfs}
\usepackage{caption}
\usepackage[noend]{algpseudocode}
\usepackage[justification=centering]{caption}
\usepackage{dsfont}
\usepackage{setspace}
\newtheorem{theorem}{Theorem}[section]

\newtheorem{lemma}[theorem]{Lemma}
\newtheorem{proposition}[theorem]{Proposition}
\theoremstyle{definition}
\newtheorem{definition}[theorem]{Definition}
\newtheorem{example}[theorem]{Example}

\newtheorem{remark}[theorem]{Remark}

\def\cX{{\cal X}}

\numberwithin{equation}{section}

\begin{document}
\makeatletter

\title{\Large {\bf A PIECEWISE CONVEXIFICATION METHOD FOR NON-CONVEX MULTI-OBJECTIVE OPTIMIZATION PROGRAMS WITH BOX CONSTRAINTS}}

\author{QIAO ZHU, LI PING TANG\thanks{Corresponding Author}, AND XIN MIN YANG}


\maketitle
\vspace*{0mm}

\vspace{2mm}

\footnotesize{
\noindent\begin{minipage}{14cm}
{\bf Abstract:}
This paper presents a piecewise convexification method for solving non-convex multi-objective optimization problems with box constraints. Based on the ideas of the $\alpha$-based Branch and Bound (${\rm \alpha BB}$) method of global optimization and the interval subdivision, a series of convex relaxation sub-multiobjective problems for this non-convex multi-objective optimization problem are firstly obtained, and these sub-problems  constitute a piecewise convexification problem of the original problem on the whole box. We then construct the (approximate, weakly) efficient solution set of this piecewise convexification problem, and use these sets to approximate the globally (weakly) efficient solution set of the original problem. Furthermore, we propose a piecewise convexification algorithm and show that this algorithm can also obtain approximate globally efficient solutions by calculating a finite subset of the efficient solution set of the multi-objective convex sub-problems only. Finally, its performance is demonstrated with various test instances.
\end{minipage}
 \\[5mm]


\hbox to14cm{\hrulefill}\par


\section{Introduction}

Multi-objective optimization plays an important role in applications and has been extensively researched in theory and practice. We note that many practical multi-objective optimization problems (MOPs) have non-convex structures, for instance, in the fields of machine learning  \cite{jin2006multi,nguyen2019multiple}, portfolio optimization \cite{babaei2015multi},
and communication networks \cite{omar2017multiobjective}, to name but a few. And how to find (approximate,weakly) efficient solutions of the non-convex multiobjective optimization problem is always an attractive research topic.

Over the past few years, a variety of effective methods for solving MOPs  have been developed.
 Unfortunately, most of the current algorithms to tackle non-convex MOPs focus on the locally optimal solutions, and finding global solutions remains a challenging task.  As we all know, the convexity assumption is a basic one in many methods which can guarantee the locally optimal solution of convex programming is eventually a global optimal solution.
In addition, it is not necessary to find a global exact solution from the perspective of practical applications and also acceptable to find a global approximate solution with a predefined quality due to the high cost of seeking a global solution.
The $\rm{\alpha BB}$ method is a well-known convex relaxation approach to explore approximate global solutions for non-convex MOPs, see  \cite{eichfelder2021general,fernandez2009obtaining,niebling2019branch,rocktaschel2020branch}.
One of the key characteristics of this method is that the maximum distance between the original non-convex function and its respective convex relaxation goes to zero as the size of the rectangle domains approaches zero. Therefore, a idea of  piecewise convexification, that is, convex relaxation on each subinterval, is naturally proposed to handle non-convex MOPs with box constraints, see \cite{eichfelder2021general,niebling2019branch}. Most of these literatures focus on algorithm design, rather than analyzing the piecewise convexification method.

The main work of this paper is to analyze this piecewise convexification method, which consists of a series of convex relaxation sub-problems of the original non-convex MOPs on different sub-interval. Combining the ${\rm \alpha BB}$ method with the interval division, we obtain a series of convex relaxation sub-multiobjective problem to piecewise approximate the original multi-objective problem on the whole interval, which can be called the piecewise convexification problem(PCP). Then, the (approximate, weakly) efficient solution set of PCP is constructed by using the partial order relationship to compare all optimal solutions of the convex relaxation sub-multiobjective problems on each sub-interval.
We prove that the (approximate, weakly) efficient solution set of PCP can be used to approximate the global approximate (weakly) efficient solutions set of  non-convex optimization problems. These theoretical results are well-structured. In the numerical experiments, the defined solution set of PCP is slightly adjusted according to the solving method, and it is proved that the algorithm can find the globally approximate optimal solution by calculating finite sub-sets of the multi-objectives sub-problems.


This paper is organized as follows. Section 2 summarizes some key notions and properties for MOPs. In addition, we recall the $\alpha{\rm BB}$ method and the interval division. In Section 3, a piecewise convexification method is discussed  for non-convex MOPs with box constraints. And we establish its convergence of (approximate) efficient solution set. Furthermore, we design a new algorithm that generates approximate (weakly) efficient solutions and  present its convergence in Section 4. Finally, we apply the algorithm to several test instances in Section 5.

\section{Preliminaries}
In this paper, we set $[p]:=\{1,\cdots,p\}$. Let $y^1, y^2\in \mathbb{R}^p$,  we shall use the notation $y^{1}\prec y^{2}$ to indicate $y_{i}^{1}<y_{i}^{2}$ for all $i\in[p]$. Moreover, $y^{1}\preceq y^{2}$ indicates $y_{i}^{1}\leq y_{i}^{2}$ for all $i\in [p]$ and $y^{1}\neq y^{2}$,  whereas  $y^{1}\leqq y^{2}$ shows $y_{i}^{1}\leq y_{i}^{2}$ for all $i\in[p]$. Let $\mathbb{R}_{++}^p=\{y\in\mathbb{R}^p: y\succ 0\}$, and $ \mathbb{R}_{+}^p=\{y\in\mathbb{R}^p: y\geqq 0\}$.

{\bf \subsection{Basic Definitions and Properties}}
In this article, we consider the following multi-objective optimization problem:
\begin{eqnarray*}
(\text{MOP})~~&&\min~~f(x)=(f_{1}(x),\cdots, f_{p}(x))\\
&&{\rm s. t.}~  x\in X,
\end{eqnarray*}
where $f_i: \mathbb{R}^m\rightarrow \mathbb{R}$ is twice continuously differentiable function for any $i\in[p]$, which is non-convex on $X$ in generally and $X:=[a, b]:=\prod\limits_{i=1}^{m}[a_i,b_i]$ satisfied $a\prec b$ with $a=(a_1,\cdots,a_m)$ and $b=(b_1,\cdots,b_m)$. Thus $x\in X$ implies that $x_{i}\in [a_i,b_{i}]$ for any $i$. Then we give the following concept of the (weakly) efficient solution with respect to (MOP).
\begin{definition}{\cite{ehrgott2005multicriteria}} A point $\hat{x}\in X$ is called (weakly) efficient solution for (MOP), if there is no other $x\in X$ such that $(f(x)\prec f(\hat{x}))$ $f(x)\preceq f(\hat{x})$.
\end{definition}
The set of all (weakly) efficient solutions for (MOP) is denoted by
\begin{align*}
&\cX_{E}:=\left\{x\in X: ~~\nexists~ y\in X~~ s.t.~~f(y)\preceq f(x)\right\},\\
&\cX_{wE}:=\left\{x\in X: ~~\nexists~ y\in X~~ s.t.~~f(y)\prec f(x)\right\}.\end{align*}
If $\hat{x}$ is an (weakly) efficient solution, then $f(\hat{x})$ is called (weakly) nondominated point. Next, we recall the definition of convex multi-objective optimization problem and its a basic property.
\begin{definition}{\cite{pardalos2017non}}
If $X$ is convex, and $f_{i}$ is convex function for any $i\in[p]$, then (MOP) is called convex multi-objective optimization problem.
\end{definition}
We will summarize some basic results for weighted sum method, which are needed in next sections.
\begin{lemma} {\cite{ehrgott2005multicriteria}}\label{lemma:CC} Suppose that $\hat{x}$ is an optimal solution of the weighted sum optimization problem:
\begin{eqnarray*}
\min\limits_{x\in \mathcal{X}}\sum_{k=1}^{p}\lambda_{k}f_{k}(x),
\end{eqnarray*}
with $\lambda\in\mathbb{R}_{+}^{p}\setminus\{0\}$, then $\hat{x}\in \cX_{wE}$. Furthermore, if $\lambda\in\mathbb{R}_{++}^{p}$, then $\hat{x}\in \cX_{E}$.
\end{lemma}
The above lemma shows that any optimal solution of the weighted
sum optimization problem  is the (weakly) efficient solution of (MOP) without any assumption. However, the reverse is not necessarily true without some convexity assumption, see the following lemma.
\begin{lemma}{\cite{ehrgott2005multicriteria}}\label{lemma:BB}
For convex MOPs, any weakly efficient solution $\bar{x}$, there exists $\bar{\lambda}\in\mathbb{R}^p_{+}\setminus\{0\}$ such that $\bar{x}$ is the optimal solution of the weighted sum optimization problem $\sum\limits_{k=1}^{p}\bar{\lambda}_{k}f_{k}(x)$.
\end{lemma}
Lemmas \ref{lemma:CC} and \ref{lemma:BB} indicate that under convexity assumption, every weakly efficient solution is  obtained by solving a weighted sum problem. Moreover, we establish a crucial property of (weakly) efficient solution, which will be applied to develop convergence properties of our piecewise convexification method.
\begin{lemma}\label{lemma:2.6} Let $X=[a, b]\subset\mathbb{R}^{m}$ and $f(x)$ be a continuous function, then

(i) For any $x\notin\cX_{E}$ there is ${y}_{x}\in\cX_{E}$ such that $f({y}_{x})\preceq f({x})$;

(ii) For any $x\notin \cX_{wE}$, there exists $y_{x}\in \cX_{wE}$ such that $f(y_{x})\prec f(x)$.
\end{lemma}
\begin{proof}
The proof is easy by Theorem 2.21 in \cite{ehrgott2005multicriteria} and Proposition 4.10 in \cite{Luc1989Lecture}.
\end{proof}

In the following, we give a brief overview of some ideas of the $\alpha{\rm BB}$ method and the interval division.
{\bf \subsection{The $\alpha{\rm BB}$ Method}}
\noindent
The $\alpha{\rm BB}$ method is a branch and bound method for solving non-convex problems in global optimization  \cite{adjiman1998globalII,adjiman1998global,androulakis1995alphabb}, which constructs a convex lower estimation function.

For $X=[a,b]=\prod\limits_{i=1}^{m}[a_{i},b_{i}]$ with $a\prec b$. Let $f: X\rightarrow \mathbb{R}$ be a real-valued twice continuously differentiable function. A relaxation function $F: X\rightarrow \mathbb{R}$ of $f$ by the idea of the $\alpha$BB method is  defined in \cite{androulakis1995alphabb} as follows:
\begin{align*}
F(x)=f(x)+\frac{\alpha}{2}\sum\limits_{i=1}^{m}(a_{i}-x_{i})(b_{i}-x_{i}),
\end{align*}
where $\alpha$ is a parameter and $x_{i}\in [a_{i}, b_{i}]$ for any $ i$.  $\frac{\alpha}{2}\sum\limits_{i=1}^{m}(a_{i}-x_{i})(b_{i}-x_{i})$ can be regarded as the error term between the relaxation function $F$ and $f$, which is completely determined by the interval where the variable is located. In particular, it is easy to verify that $F$ is convex on $X$ from \cite{adjiman1998global} when taking
\begin{align}\label{2022032901}
\alpha\geq \max\{0,-\min\limits_{x\in X}\lambda_{min}(x)\},
\end{align}
where $\lambda_{min}(x)$ is the minimum eigenvalue of Hessian matrix to $f$ on $x$, which can be calculated by the interval arithmetic method \cite{floudas2013deterministic,hansen2003global}.
\vspace{0.2cm}

{\bf\subsection{Interval Division}}
\noindent
From the form of the ${\rm \alpha BB}$ method, dividing the interval can better approximate the original function.   Analogously to \cite{eichfelder2016modification}, we also use the longest edge of the interval to subdivide it, that is, for a given box $X=[a,b]=\prod\limits_{i=1}^{m}[a_{i},b_{i}]$,  the longest edge of $X$ is defined by
\begin{align*}
l=\min\left\{i\in[m]:~i\in\arg\max\limits_{j\in[m]} (b_{j}-a_{j})\right\},
\end{align*}
and according to $l$ to subdivide $X$, two subsets of $X$ have the following form:
\begin{align*}
Y^{1}=\prod\limits_{i=1,i\neq l}^{m}[a_{i}, b_{i}]\times \left[a_{l}, \frac{a_{l}+b_{l}}{2}\right],~~
Y^{2}=\prod\limits_{i=1,i\neq l}^{m}[a_{i}, b_{i}]\times \left[ \frac{a_{l}+b_{l}}{2}, b_{l}\right].
\end{align*}
Clearly, $Y^{i}\subset X (i=1,2)$ and $X=Y^{1}\cup Y^{2}$. As shown in \cite{eichfelder2016modification}, the interval division operator $\text{ID}: X\times \mathbb{N}_0\rightarrow 2^{X}$ is recursively defined by
\begin{align*}
&\text{ID}(X,0):=\{X\},~~ \text{ID}(X,1):=\{Y^{1},Y^{2}\}, \\ &\text{ID}(X,t):=\text{ID}(Y^{1},t-1)\cup \text{ID}(Y^{2},t-1),
\end{align*}
where $t$ is the interval division times of $X$. Apparently for any $t\in\mathbb{N}_0$ it holds that $|\text{ID}(X,t)|=2^{t}$. Thus, for simplicity of presentation, let
\begin{align*}
\mathbf{Y}^{t}:=\text{ID}(X,t)=\{Y^{1},Y^{2},\cdots,Y^{2^t}\},
\end{align*}
where $Y^{k_t}\subset X~(k_{t}\in[2^t])$ and $\bigcup\limits_{{k_t}=1}^{2^t}Y^{k_t}=X$.

In this paper, $\mathbf{Y}^t$ is called a subdivision of $X$. For simplicity,  we may abbreviate the subinterval $Y^{k_t}$ of $X$ as $Y^{k_t}=[a^{k_t},b^{k_t}]=\prod\limits_{i=1}^{m}[a_{i}^{k_t}, b_{i}^{k_t}]$ and define the length of the subdivision $\mathbf{Y}^t$ of $X$ by
\begin{eqnarray*}
|T(\mathbf{Y}^{t})|=\max\limits_{k_t\in [2^t]}\left\{\|a^{k_t}-b^{k_t}\|_{2}^{2}\right\}=\max\limits_{k_t\in [2^t]}\left\{\sum\limits_{i=1}^{m}(b_{i}^{k_t}-a_{i}^{k_t})^{2}\right\}.
\end{eqnarray*}
Obviously, $|T(\mathbf{Y}^{t_1})|\leq |T(\mathbf{Y}^{t_2})|$ when $t_1>t_2$.  Combining with the ${\rm \alpha BB}$ method, the smaller length of the subdivision $\mathbf{Y}^t$, the smaller error term, that is,  $F$ is closer to the original function $f$. Since $t$
uniquely determines the subdivision $\mathbf{Y}^{t}$ of X, thus for simplicity $|T(\mathbf{Y}^{t})|$ can be abbreviated as $|T(t)|$. Moreover, it follows from Lemma 5 in \cite{eichfelder2016modification} that $|T(t)|\leq \|b-a\|_{2}^{2}\cdot\left(1-\frac{3}{4m}\right)^t$, which implies that $|T(t)|\rightarrow 0$ as $t\rightarrow \infty$.
\vspace{0.2cm}

\section{Piecewise Convexification of (MOP)}
In this section, combining the ${\rm \alpha BB}$ method and the interval division, a piecewise convexification problem for ${\rm(MOP)}$ is obtained. Then we design its (approximate, weakly) efficient solution sets, which is our key tools to approximate the approximate (weakly) efficient solution set of ${\rm(MOP)}$.

{\bf\subsection{Piecewise Convex Relaxation for (MOP)}}
\noindent
In this subsection, we state with an overview of the convex relaxation for (MOP) on each sub-interval according to the $\alpha{\rm BB}$ method, and then discuss some properties of parameters for this relaxation problem. Finally, the relationship between the sets of solutions for this problem is proposed.

For any given $t\in\mathbb{N}_0$ it uniquely determines a subdivision $\mathbf{Y}^t$ of $X$.  Similar to \cite{niebling2019branch},  the $\alpha {\rm BB}$ convex relaxation for (MOP) on subset $Y^{k_t}\in\mathbf{Y}^t$ for any $k_t\in\{1,\cdots,2^t\}$ is constructed as follows
\begin{eqnarray*}
({\rm CMOP})^{k_t}~~&&\min~~f^{k_t}(x)=(f_{1}^{k_t}(x),\cdots,  f_{p}^{k_t}(x))\\\label{eqnarray:1-12-2}
&&{\rm s. t.}~  x\in Y^{k_t},\notag
\end{eqnarray*}
where\vspace{-1mm}
\begin{eqnarray}\label{eqn:(7)}
f_{j}^{k_t}(x)=f_{i}(x)+\frac{\alpha_{j}^{k_t}}{2}\sum\limits_{i=1}^{n}(a_{i}^{k_t}-x_{i})(b_{i}^{k_t}-x_{i}), j=[p],
\end{eqnarray}
$\alpha_{j}^{k_t}=\max\{0, -\min\limits_{x\in Y^{k_t}} \lambda_{\min}^{j}(x)\}$ and $\alpha_{j}=\max\{0, -\min\limits_{x\in X}\lambda_{\min}^{j}(x)\}$. Here $\lambda_{\min}^{j}(x)$ indicates the minimum eigenvalue of $\nabla^2 f_j(x), j=[p]$. The definition of $\alpha_{j}^{k_t}$ indicates that $f_{j}^{k_t}$ is convex on $Y^{k_t}$. Thus, $({\rm CMOP})^{k_t}$ is regarded as the local convex relaxation sub-multiobjective optimization problem of (MOP) on $Y^{k_t}\subset X$. All local convex relaxation sub-multiobjective optimization problems for the subdivision $\mathbf{Y}^{t}$ are constructed the piecewise convexification problem of (MOP) on the whole set $X$.

From the definitions of $\alpha_{j}^{k_t}$ and $\alpha_{j}$, it follows that $\alpha_{j}^{k_t}\leq \alpha_{j}$ for any $j=[p]$ and $k_t\in[2^t]$.
This indicates that a relaxation function of piecewise convexification is closer to the original function than that direct convexification on the whole interval by using the ${\rm \alpha BB}$ method, which can be illustrated in  Figure \ref{fig:0-1}.
\begin{figure}[h]
\captionsetup{font={scriptsize}}
\centering
\includegraphics[width=10.5cm]{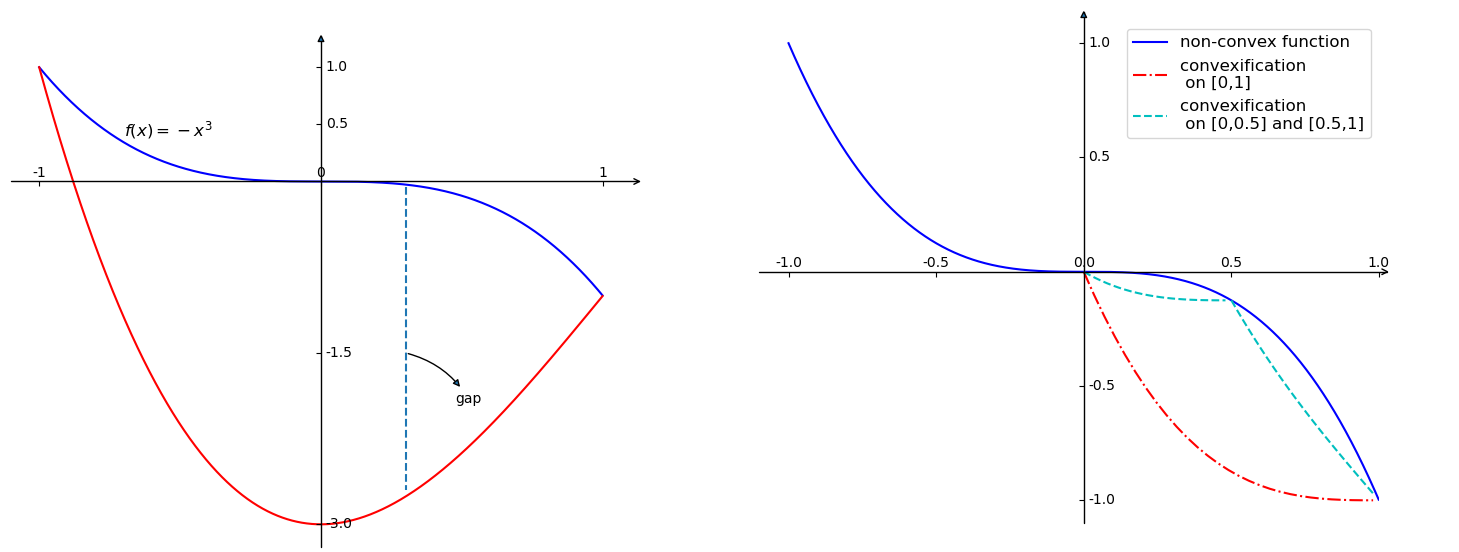}
\caption{\scriptsize{The process of convexification relaxation for $y=-x^3$ in $[-1,1]$. }}\label{fig:0-1}
\end{figure}
Note that Figure \ref{fig:0-1} (a) shows the convexification result in blue on the whole interval. However, Figure \ref{fig:0-1} (b) demonstrates the piecewise convexification results for the different subdivisions and also suggests that the number of interval subdivisions affects the results of convexification, i.e., the number of divisions directly affects the effect of the approximation.

In this paper, let $\alpha=\max\left\{\alpha_{1},\cdots,  \alpha_{p}\right\}$. Incorporating with these definitions, this implies that
\begin{eqnarray}\label{eqna: 3.2}
\alpha\geq\alpha_{j}\geq\alpha_{j}^{k_t}, \forall j=[p]; k_t=[2^t], t\in \mathbb{N}_0.
\end{eqnarray}
From the above analysis, one can obtain the difference between $f_{i}^{k_t}$ and $f_{i}$.
\begin{proposition}{\cite{androulakis1995alphabb}}\label{proposition:3.1}
For any given subdivision $\mathbf{Y}^{t}$ of $X$ and any subinterval $Y^{k_t}\in\mathbf{Y}^{t}$,  one has

(i) $f_{j}^{k_t}(x)$ is a local lower bound function of $f_{j}(x)$ on $Y^{k_t}$, that is,
$$f_{j}^{k_t}(x)\leq f_{j}(x), \forall x\in Y^{k_t}, j\in[p].$$

(ii) For any  $x\in Y^{k_t}$, we have
$\sum\limits_{i=1}^{m}(a_{i}^{k_t}-x_{i})(b_{i}^{k_t}-x_{i})\geq -\left\|\frac{b^{k_t}-a^{k_t}}{2}\right\|_{2}^2.$
\end{proposition}

We are now in the position to define the (weakly) efficient solution set of this piecewise convexification problem. Our analysis heavily relies on this definition way, which will help us to approximate the globally approximate (weakly) efficient solution set of non-convex MOPs.  Let $\cX_{wE_{ap}^{k_t}}$ and $\cX_{E_{ap}^{k_t}}$ be weakly efficient set and efficient set of $({\rm CMOP})^{k_t}$, respectively. The set of (weakly) efficient solutions of the piecewise convexification problem with respect to the subdivision $\mathbf{Y}^t$ of $X$  is defined by
\begin{align}
&\cX_{wE_{ap}}(\mathbf{Y}^t)=\left\{x\in\bigcup\limits_{k_t=1}^{2^t}\cX_{wE_{ap}^{k_t}}: ~\nexists~ y\in\bigcup\limits_{k_t=1}^{2^t}\cX_{wE_{ap}^{k_t}}~ s.t.~f(y)\prec f(x)\right\},\label{WE:1}\\
&\cX_{E_{ap}}(\mathbf{Y}^t)=\left\{x\in\bigcup\limits_{k_t=1}^{2^t}\cX_{E_{ap}^{k_t}}: ~\nexists~ y\in\bigcup\limits_{k_t=1}^{2^t}\cX_{E_{ap}^{k_t}}~s.t.~f(y)\preceq f(x)\right\}. \label{E:1}
\end{align}
These indicate that we need to check every (weakly) efficient solution for each $({\rm CMOP})^{k_t}$ through the partial order relation of the original problem. Since $t$ uniquely determines the subdivision $Y^{t}$ of $X$, thus  $\cX_{wE_{ap}}(\mathbf{Y}^t)$ and $\cX_{E_{ap}}(\mathbf{Y}^t)$ can be abbreviated as $\cX_{wE_{ap}}(t)$ and $\cX_{E_{ap}}(t)$, respectively. It can easily be verified that solution set  $\cX_{wE_{ap}}(t)$ is not empty for any $t\in N$. So far we have not proved that $\cX_{E_{ap}}(t)$ is a non-empty set, or given a counter-example to show $\cX_{E_{ap}}(t)=\emptyset$. However, from the following algorithm, only finite solutions are found for each $\cX_{E_{ap}^{k_t}}$, that is, $\cX_{E_{ap}^{k_t}}$ is a compact set, then the set $\cX_{E_{ap}}(t)$ is non-empty in the actual calculation. Thus, in this paper we always assume that $\cX_{E_{ap}}(t)\neq \emptyset$ for any $t\in N$.
\begin{remark}\label{remark:3.1} From the definition of $\cX_{wE_{ap}}(t)$, it is easy to see that

(i) If $\hat{x}\notin \cX_{wE_{ap}}(t)$, then there exists $k_{t}\in\{1,\cdots, 2^t\}$ such that $\hat{x}\notin \cX_{wE_{ap}^{k_t}}$, or $\hat{x}\in \cX_{wE_{ap}^{k_t}}$ satisfying $f(\hat{y})\prec f(\hat{x})$ for some  $\hat{y}\in \bigcup\limits_{k_t=1}^{2^t}\cX_{wE_{ap}^{k_t}}$.

(ii) If $\hat{x}\in\cX_{wE_{ap}}(t)$, then there must exist  $k_t$ satisfying $\hat{x}\in\cX_{wE_{ap}^{k_t}}$.
\end{remark}
This remark also holds for $\cX_{E_{ap}}(t)$.

{\bf \subsection{Convergence
of the Solution Set for the Piecewise Convexification Problem}}
\noindent
In this subsection, we develop the convergence
of the solution set for the piecewise convexification problem. More precisely, we investigate some relationships between the set of approximate (weakly) efficient solutions of the original problem and the set of (weakly) efficient solutions of the convexification problem.

Let $e=(1,\cdots,1)\in \mathbb{R}^p$ be all-ones vector. For ${\varepsilon}>0$, let $\cX_{wE}^{\varepsilon}$ and  $\cX_{E}^{\varepsilon}$ denote the set of all $\varepsilon$-weakly efficient solutions and of all $\varepsilon$-efficient solutions for (MOP), respectively, i.e.,
\begin{align*}
&\cX_{wE}^{\varepsilon}=\left\{x\in X: ~~\nexists~ y\in X~~ s.t.~~f(y)+{\varepsilon}e\prec f(x)\right\},\\
&\cX_{E}^{\varepsilon}=\left\{x\in X: ~~\nexists~ y\in X~~ s.t.~~f(y)+{\varepsilon}e\preceq f(x)\right\}.
\end{align*}
If $\hat{x}\in X$ is a $\varepsilon$-(weakly) efficient solution, then $f(\hat{x})$ is a $\varepsilon$-(weakly) non-dominated point in image space. From the above definitions, we discuss the relation between $\cX_{wE_{ap}}(t)$ and $\cX_{wE}^{\varepsilon}$ which plays a vital role in determining convergence of set of (weakly) efficient solutions for the piecewise convexification problem.

\begin{theorem}\label{thm:4.2} For any $\varepsilon>0$, there exists ${t}_{\varepsilon}\in\mathbb{N}_0$ such that
\begin{align*}
\cX_{wE_{ap}}(t)\subseteq \cX_{wE}^{\varepsilon}, ~~~~\forall~ t>{t}_{\varepsilon},
\end{align*}
where $\cX_{wE_{ap}}(t)$ is the weakly efficient solution set
of the piecewise convexification problem with respect to the subdivision $\mathbf{Y}^{t}$ of $X$.
\end{theorem}
\begin{proof}
From the fact that $|T(t)|\rightarrow 0$ as $t\rightarrow 0$ and the definition of $\alpha$, it follows that for any $\varepsilon>0$ there exists ${t}_{\varepsilon}\in\mathbb{N}$ satisfying
\begin{eqnarray}\label{eqnarry:1-11-2}
 \max\limits_{k_t\in[2^t]}\frac{\alpha}{2}\left\|\frac{b^{k_t}-a^{k_t}}{2}\right\|^2<\varepsilon, \forall t>{t}_{\varepsilon}.
\end{eqnarray}
Next, we prove $\cX_{wE_{ap}}(t)\subseteq\cX_{wE}^{\varepsilon}$ for any $t>{t}_{\varepsilon}$. By contradiction, we assume that there exists $t_0>{t}_{\varepsilon}$ satisfying $\cX_{wE_{ap}}(t_0)\nsubseteq \cX_{wE}^{\varepsilon}$, that is, one can find $\hat{x}\in X$ satisfying $\hat{x}\in \cX_{wE_{ap}}(t_0)$ and $\hat{x}\notin \cX_{wE}^{\varepsilon}$. Here, $\cX_{wE_{ap}}(t_0)$ is a solution set w.r.t. the subdivision $\mathbf{Y}^{t_0}$ of $X$ where $\mathbf{Y}^{t_0}=\{Y^{1},\cdots,Y^{2^{t_0}}\}$. Thus $\hat{x}\notin \cX_{wE}^{\varepsilon}$ implies that there exists $\hat{y}\in Y^{j_{t_0}}\subseteq X$ for some $j_{t_0}\in\{1,\cdots,2^{t_0}\}$ such that
\begin{eqnarray}\label{eqna:4.1}
f(\hat{y})+{\varepsilon}e\prec f(\hat{x}).
\end{eqnarray}
 If $\hat{y}\in \cX_{wE_{ap}^{j_{t_0}}}\subset Y^{j_{t_0}}$, then $f(\hat{y})\prec f(\hat{x})$ from $\varepsilon>0$ and (\ref{eqna:4.1}).  It is a contradiction to $\hat{x}\in \cX_{wE_{ap}}(t_0)$. Otherwise, if $\hat{y}\in Y^{j_{t_0}}\setminus \cX_{wE_{ap}^{j_{t_0}}}$, then combining with Lemma \ref{lemma:2.6} (ii), there exists  $\nu\in \cX_{wE_{ap}^{j_{t_0}}}$ such that  $f^{j_{t_0}}(\nu)\prec f^{j_{t_0}}(\hat{y})$. From (\ref{eqna:4.1}) and  definition of $f_{i}^{j_{t_0}}$, it is easy to verify that
$$f^{j_{t_0}}(\nu)\prec f^{j_{t_0}}(\hat{y})\leqq f(\hat{y})\prec f(\hat{x})-{\varepsilon}e.$$
By the definition of $f^{j_{t_0}}$, we have
\begin{align}\label{align:1-11-1}
f(\nu)+\boldsymbol{\tau}+{\varepsilon}e\prec f(\hat{x}),\end{align}
where $\boldsymbol{\tau}=(\tau_1,\cdots,\tau_p)$ and $\tau_i=\frac{\alpha_{i}^{j_{t_0}}}{2}\langle a^{j_{t_0}}-\nu, b^{j_{t_0}}-\nu\rangle$ for any $i\in[p]$. According to (\ref{eqna: 3.2}) and (\ref{eqnarry:1-11-2}),  it concludes that
\begin{align*}\frac{\alpha_{i}^{j_{t_0}}}{2}\left\|\frac{b^{j_{t_0}}-a^{j_{t_0}}}{2}\right\|^2
<\frac{\alpha}{2}\left\|\frac{b^{j_{t_0}}-a^{j_{t_0}}}{2}\right\|^2
<\varepsilon, \forall i=1,\cdots,p. \end{align*}
It follows from Proposition \ref{proposition:3.1} (ii) that
$\tau_i + \varepsilon >0, i=[p]$. It conducts $f(\nu)\prec f(\hat{x})$ by   (\ref{align:1-11-1}), which implies a contradiction to $\hat{x}\in\cX_{wE_{ap}}(t_{0})$. Thus we have derived that $\cX_{wE_{ap}}(t)\subseteq \cX_{wE}^{\varepsilon}$ for any $t>{t}_{\varepsilon}$.
\end{proof}

Similarly,  one can build the same relation between $\cX_{E_{ap}}(t)$ and $\cX_{E}^{\varepsilon}$.
\begin{theorem}\label{thm1:4.3} For any $\varepsilon>0$, there exists $t_{\varepsilon}\in\mathbb{N}_0$ such that $\cX_{E_{ap}}(t)\subseteq\cX_{E}^{\varepsilon}$ for any $t>t_{\varepsilon}$, where $\cX_{E_{ap}}(t)$ denotes the efficient solution set of the piecewise convexification problem w.r.t. the subdivision $\mathbf{Y}^{t}$ of $X$.
\end{theorem}
\begin{proof}
Using the same argument as in the proof of Theorem \ref{thm:4.2}, we can easily carry out that for any $\varepsilon$ there exists $t_{\varepsilon}\in\mathbb{N}$ satisfying (\ref{eqnarry:1-11-2}). As we have assumed in the previous section, $\cX_{E_{ap}}(t)\neq \emptyset$ for any  $t\in N$. Therefore, we assume that there exists $\hat{x}\in X$ such that $\hat{x}\in \cX_{E_{ap}}(t_0)$ and $\hat{x}\notin \cX_{E}^{\varepsilon}$ for some $t_0>t_{\varepsilon}$. Then, it follows from the assumption $\hat{x}\notin \cX_{E}^{\varepsilon}$ that there is $\hat{y}\in X$ satisfying $f(\hat{y})+ {\varepsilon}e\preceq f(\hat{x})$. Combining Lemma \ref{lemma:2.6} (i)  with (\ref{eqnarry:1-11-2}), the remainder of the argument is analogous to that in Theorem \ref{thm:4.2}, which is contrary to $\hat{x}\in \cX_{E_{ap}}(t_0)$. Therefore, $\cX_{E_{ap}}(t)\subseteq\cX_{E}^{\varepsilon}$ for any $t>t_{\varepsilon}$.
\end{proof}
Theorems \ref{thm:4.2} and \ref{thm1:4.3} show that the solution set of this piecewise convexification problem is a lower bound set of the approximate (weakly) efficient solution set of (MOP). In order to get a closer approximation, we need also to be concerned about the upper bound set of the approximate (weakly) efficient solution set of (MOP). In what follows, we study the approximate solution set of the piecewise convexification problem, which is an upper bound set of approximation (weakly) efficient solution set of (MOP).

{\bf \subsection{Convergence of the Approximate Solution Set for the Piecewise Convexification Problem}}
\noindent
In this subsection, we consider the sets of the approximate (weakly) efficient solutions of the piecewise convexification problem and establish the convergence results of the approximate solution set.

For any given subdivision $\mathbf{Y}^{t}$ of $X$.  Let $\cX_{wE_{ap}^{k_t}}^{\varepsilon}$ and $\cX_{E_{ap}^{k_t}}^{\varepsilon}$ be denoted the set of all $\varepsilon$-weakly efficient solutions and of all $\varepsilon$-efficient solutions of $({\rm CMPO})^{k_t}$ on $Y^{k_t}\in\mathbf{Y}^{t}$, respectively. Then, we define the approximate weakly efficient solution set and  approximate efficient solution set of the piecewise convexification problem w.r.t. the subdivision $\mathbf{Y}^{t}$ of $X$, respectively, i.e.,
\begin{align}
&\cX_{wE_{ap}}^{\varepsilon}(\mathbf{Y}^{t})=\left\{x\in\bigcup\limits_{k_{t}=1}^{2^t}\cX_{wE_{ap}^{k_t}}^{\varepsilon}: ~\nexists~ y\in\bigcup\limits_{k_{t}=1}^{2^t}\cX_{wE_{ap}^{k_t}}^{\varepsilon}, s.t.~ f(y)+ {\varepsilon}e\prec f(x)\right\},\label{2022031802}\\
&\cX_{E_{ap}}^{\varepsilon}(\mathbf{Y}^{t})=\left\{x\in\bigcup\limits_{k_{t}=1}^{2^t}\cX_{E_{ap}^{k_t}}^{\varepsilon}: ~\nexists~ y\in\bigcup\limits_{k_{t}=1}^{2^t}\cX_{E_{ap}^{k_t}}^{\varepsilon}, s.t. ~ f(y)+{\varepsilon}e\preceq f(x)\right\}.\label{2022031801}
\end{align}
For simplicity, $\cX_{wE_{ap}}^{\varepsilon}(\mathbf{Y}^{t})$ and $\cX_{E_{ap}}^{\varepsilon}(\mathbf{Y}^{t})$ can be abbreviated as $\cX_{wE_{ap}}^{\varepsilon}(t)$ and $\cX_{E_{ap}}^{\varepsilon}(t)$, respectively.
It is easy to show that $\cX_{wE_{ap}}^{\varepsilon}(t)\neq \emptyset$ for any number of subdivision $t$. Moreover, according to Lemma 4.12 in \cite{chen2006vector} one can conclude that $\cX_{E_{ap}}^{\varepsilon}(t)\neq \emptyset$.
With the help of the preceding concepts, we can now present some vital characters of  $\cX_{wE_{ap}}^{\varepsilon}(t)$ and $\cX_{wE_{ap}}(t)$.
\begin{proposition}\label{proposition:20210923} For any $\varepsilon>0$ there exists  $t_{\varepsilon}\in\mathbb{N}_0$ such that
\begin{align*}
\cX_{wE_{ap}}(t)\subseteq \cX_{wE_{ap}}^{\varepsilon}(t),~~ \forall~t>t_{\varepsilon}.
\end{align*}
\end{proposition}
\begin{proof}
 For any $\varepsilon>0$ there exists  $t_{\varepsilon}\in\mathbb{N}_0$ such that (\ref{eqnarry:1-11-2})  holds and $\cX_{wE_{ap}}^{\varepsilon}(t)\neq\emptyset$ for any $t$. Next, we assume that $\cX_{wE_{ap}}(t_0)\nsubseteq \cX_{wE_{ap}}^{\varepsilon}(t_0)$ for some $t_0>t_{\varepsilon}$, then there exists $\hat{x}\in X$ such that $\hat{x}\in\cX_{wE_{ap}}(t_0)$ and $\hat{x}\notin\cX_{wE_{ap}}^{\varepsilon}(t_0)$. The assumption $\hat{x}\in\cX_{wE_{ap}}(t_0)$ indicates that $\hat{x}\in \bigcup\limits_{k_t=1}^{2^{t_0}}\cX_{wE_{ap}^{k_t}}$. In fact, $\cX_{wE_{ap}^{k_{t}}}\subseteq\cX_{wE_{ap}^{k_{t}}}^{\varepsilon}$ holds for any $\varepsilon>0$ and $k_{t}$. Then $\hat{x}\in\bigcup\limits_{k_t=1}^{2^{t_0}}\cX_{wE_{ap}^{k_t}}^{\varepsilon}$. By Remark \ref{remark:3.1} (i), the assumption $\hat{x}\notin\cX_{wE_{ap}}^{\varepsilon}(t_0)$  implies that there exists  $\hat{y}\in\cX_{wE_{ap}^{j_{t_0}}}^{\varepsilon}\subset\bigcup\limits_{k_t=1}^{2^{t_0}}\cX_{wE_{ap}^{k_t}}^{\varepsilon}$ for some $j_{t_0}\in[2^{t_0}]$ such that
\begin{eqnarray}\label{eqna:4.8}
f(\hat{y})+ {\varepsilon}e\prec f(\hat{x}).
\end{eqnarray}
It is easy to show that if $\hat{y}\in\cX_{wE_{ap}^{j_{t_0}}}\subset\cX_{wE_{ap}^{j_{t_0}}}^{\varepsilon}$, then it concludes $\hat{x}\notin\cX_{wE_{ap}}(t_0)$ since $f(\hat{y})\prec f(\hat{x})$ from (\ref{eqna:4.8}). Otherwise, if $\hat{y}\in\cX_{wE_{ap}^{j_{t_0}}}^{\varepsilon}\setminus \cX_{wE_{ap}^{j_{t_0}}}$, then as Lemma \ref{lemma:2.6} (ii)
there exists $\nu\in\cX_{wE_{ap}^{j_{t_0}}}\subset\bigcup\limits_{k_n=1}^{2^{t_0}}\cX_{wE_{ap}^{k_t}}$ such that $f^{j_{t_0}}(\nu)\prec f^{j_{t_0}}(\hat{y})$. From the definition of $f_{i}^{j_{t_0}}$ and (\ref{eqna:4.8}), it can be rewritten as
\begin{eqnarray}\label{eqnarray:2}
f(\nu)+\boldsymbol{\tau}+ {\varepsilon}e\prec f(\hat{x}),\end{eqnarray}
where $\boldsymbol{\tau}=(\tau_1,\cdots, \tau_p)$ and $\tau_{i}=\frac{\alpha_{i}^{j_{t_0}}}{2}\langle a^{j_{t_0}}-\nu, b^{j_{t_0}}-\nu\rangle$. Moreover, one can yield that $\tau_{i}+\varepsilon>0$ for any $i$ by (\ref{eqnarry:1-11-2}). Hence, (\ref{eqnarray:2}) conducts  $f(\nu)\prec f(\hat{x})$. It implies $\hat{x}\notin \cX_{wE_{ap}}(t_0)$ for $\nu\in \cX_{wE_{ap}^{j_{t_0}}}\subseteq \bigcup\limits_{k_n=1}^{2^{t_0}}\cX_{wE_{ap}^{k_t}}$. Apparently, this is a contradiction to the assumption $\hat{x}\in \cX_{wE_{ap}}(t_0)$, thus the theorem holds.
\end{proof}

\begin{proposition} For any $\varepsilon>0$ there exists $t_{\varepsilon}\in\mathbb{N}_0$  such that
\begin{align*}
\cX_{E_{ap}}(t)\subseteq\cX_{E_{ap}}^{\varepsilon}(t), ~~\forall~t>t_{\varepsilon}.
\end{align*}
\end{proposition}
\begin{proof}
We shall adopt the same procedure as in the proof of Proposition \ref{proposition:20210923}.
\end{proof}

In what follows, we show a vital fact that the approximate weakly efficient solutions set of the piecewise convexification problem is an upper bound set of the weakly efficient solution set of the original multi-objective optimization problem.
\begin{theorem}\label{theorem:4.6} For any $\varepsilon >0$ there exists $t_{\varepsilon}\in \mathbb{N}_0$ such that $\cX_{wE}\subset \cX_{wE_{ap}}^{\varepsilon}(t)$ for any $t>t_{\varepsilon}$, where $\cX_{wE_{ap}}^{\varepsilon}(t)$ represents the $\varepsilon$-weakly efficient solution set of the piecewise convexification problem with respect to the subdivision $\mathbf{Y}^{t}$ of $X$.
\end{theorem}
\begin{proof}
For any $\varepsilon>0$, we can also take $t_{\varepsilon}$ as in (\ref{eqnarry:1-11-2}). In the following, we suppose that there exist $t_0>t_{\varepsilon}$ and $\hat{x}\in X$ such that $\hat{x}\in\cX_{wE}$ and $\hat{x}\notin \cX_{wE_{ap}}^{\varepsilon}(t_0)$.

Without loss of generality, we assume that $\hat{x}\in Y^{k_{t_0}}$ for some $k_{t_0}$. From Remark \ref{remark:3.1}(i), $\hat{x}\notin\cX_{wE_{ap}}^{\varepsilon}(t_0)$ has two cases. The first of which is that $\hat{x}\in \cX_{wE_{ap}^{k_{t_0}}}^{\varepsilon}$ and there is $\hat{y}\in \bigcup\limits_{k_{n}=1}^{2^{t_0}}\cX_{wE_{ap}^{k_t}}^{\varepsilon}$ satisfying $f(\hat{y})+ {\varepsilon}e\prec f(\hat{x})$. Obviously, this case is a contradiction to the assumption $\hat{x}\in \cX_{wE}$. The other case is $\hat{x}\notin\cX_{wE_{ap}^{k_{t_0}}}^{\varepsilon}$, which shows that there exists  $\hat{y}\in Y^{k_{t_0}}$ satisfying $f^{k_{t_0}}(\hat{y})+ {\varepsilon}e\prec f^{k_{t_0}}(\hat{x})$.
Similarly, incorporating (\ref{align:1-11-1}) with Proposition \ref{proposition:3.1} (ii),  we can get
$f(\hat{y})\prec f(\hat{x}),$
which is a contradiction to the assumption $\hat{x}\in \cX_{wE}$ since $\hat{y}\in Y^{k_{t_0}}$. Hence, this statement is proved.
\end{proof}
\begin{remark}
It is not difficult to verify that the above statement also holds for efficient solutions , i.e., $\cX_{E}\subseteq\cX_{E_{ap}}^{\varepsilon}(t)$ under the same assumptions.
\end{remark}
Furthermore, we consider the relationship between two sets of the approximate solutions, that is, the relationship between $\cX_{wE}^{\delta}$ and $\cX_{wE_{ap}}^{\varepsilon}(t)$.
\begin{theorem}\label{theorem:4.8} For any $\varepsilon,\delta$ with $0<\delta<\varepsilon$ there exists $t_{\varepsilon,\delta}\in\mathbb{N}_0$ such that $\cX_{wE}^{\delta}\subseteq\cX_{wE_{ap}}^{\varepsilon}(t)$ for any $t>t_{\varepsilon,\delta}$, where $\cX_{wE}^{\delta}$ is the $\delta$-weakly efficient solution set of (MOP) and $\cX_{wE_{ap}}^{\varepsilon}(t)$ is the ${\varepsilon}$-weakly efficient solution set of the piecewise convexification problem w.r.t. the subdivision $\mathbf{Y}^{t}$ of $X$, respectively.
\end{theorem}
\begin{proof}
Employing the definition of $\alpha$ and $|T(t)|\rightarrow 0$ as $t\rightarrow 0$, one can easily conclude that for any $\varepsilon >0$ and $0<\delta<\varepsilon$, there exists  $t_{\varepsilon,\delta}\in\mathbb{N}_0$ such that
\begin{eqnarray}\label{eqna:4.9}
\max\limits_{k_t\in[2^t]} \frac{\alpha}{2}\left\|\frac{b^{k_t}-a^{k_t}}{2}\right\|^2< (\varepsilon-\delta), \forall t>t_{\varepsilon,\delta}.
\end{eqnarray}
Our task now is to proof $\cX_{wE}^{\delta}\subseteq\cX_{wE_{ap}}^{\varepsilon}(t)$ for any $t>t_{\varepsilon,\delta}$. Obviously, the remainder of the argument is analogous to that in Theorem \ref{theorem:4.6}.
\end{proof}
\begin{remark} In the same way,  we infer that for any $\varepsilon, \delta$ with $0<\delta<\varepsilon$ there exists $t_{\varepsilon,\delta}\in\mathbb{N}_0$ such that $\cX_{E}^{\delta}\subseteq\cX_{E_{ap}}^{\varepsilon}(t)$ for any $t>t_{\varepsilon,\delta}$.
\end{remark}
We now turn to develop the convergence of $\cX_{E_{ap}}^{\delta}(t)$ and $\cX_{wE_{ap}}^{\delta}(t)$ and point out that these results are more useful in the actual calculation.
\begin{theorem}\label{theorem:4.9} For any $\varepsilon, \delta$ with $0<\delta<\varepsilon$ there exists $t_{\varepsilon,\delta}\in\mathbb{N}_0$ such that $\cX_{wE_{ap}}^{\delta}(t)\subseteq\cX_{wE}^{\varepsilon}$ for any $t>t_{\varepsilon,\delta}$, where $\cX_{wE}^{\varepsilon}$ is the $\varepsilon$-weakly efficient solution set of (MOP) and $\cX_{wE_{ap}}^{\delta}(t)$ is the ${\delta}$-weakly efficient solution set of the piecewise convexification problem w.r.t. the subdivision $\mathbf{Y}^{t}$ of $X$, respectively.
\end{theorem}
\begin{proof}
Similarly, we can carry out (\ref{eqna:4.9}). If the statement was false, then there exist  $t_0>t_{\varepsilon,\delta}$ and $\hat{x}\in X$ satisfying $\hat{x}\in\cX_{wE_{ap}}^{\delta}(t_0)$ and $\hat{x}\notin\cX_{wE}^{\varepsilon}$. Then $\hat{x}\in\cX_{wE_{ap}}^{\delta}(t_0)$ implies that $\hat{x}\in\cX_{wE_{ap}^{k_{t_0}}}^{\delta}\subseteq \bigcup\limits_{k_t=1}^{2^{t_0}}\cX_{wE_{ap}^{k_{t}}}^{\delta}$ for some $k_{t_0}$ and there does not exist $y\in\bigcup\limits_{k_n=1}^{t_0}\cX_{wE_{ap}^{k_{n}}}^{\delta}$ satisfying  $f(y)+{\delta}e\prec f(\hat{x})$.
However, since $\hat{x}\notin\cX_{wE}^{\varepsilon}$, one can take $\hat{y}\in X$ satisfying
\begin{eqnarray}\label{eqna:4.13}
f(\hat{y})+{\varepsilon}e\prec f(\hat{x}).
\end{eqnarray}
In the same way, there exists $j_{t_0}$ such that $\hat{y}\in Y^{j_{t_0}}\subset \mathbf{Y}^{t}$,  which only has two cases. The first case is $\hat{y}\in\cX_{wE_{ap}^{j_{t_0}}}^{\delta}\subset Y^{j_{t_0}}$. According to $0<\delta<\varepsilon$, (\ref{eqna:4.13}) indicates that $f(\hat{y})+{\delta}e\prec f(\hat{x})$, which shows that $\hat{x}\notin\cX_{wE_{ap}}^{\delta} (t_0)$. This is a contradiction to the assumption $\hat{x}\in\cX_{wE_{ap}}^{\delta}(t_0)$.

However, the second case is $\hat{y}\in Y^{j_{t_0}}\setminus\cX_{wE_{ap}^{j_{t_0}}}^{\delta}$. From Lemma \ref{lemma:2.6} (ii) there exists $\nu\in\cX_{wE_{ap}^{j_{t_0}}}\subset\cX_{wE_{ap}^{j_{t_0}}}^{\delta}$ such that $f^{j_{t_0}}(\nu)\prec f^{j_{t_0}}(\hat{y})$. From the definition of $f_{i}^{k_n}$ and (\ref{eqna:4.13}), it is equivalent to
\begin{eqnarray}\label{eqna:4.14}
f(\nu)+{\delta}e+\boldsymbol{\tau}\prec f(\hat{x}),
\end{eqnarray}
where $\boldsymbol{\tau}=(\tau_1,\cdots,\tau_p)$ and $\tau_{i}=\frac{\alpha_{i}^{j_{t_0}}}{2}\langle a^{j_{t_0}}-\nu, b^{j_{t_0}}-\nu\rangle+\varepsilon-\delta$. And since $t_0>{t}_{\varepsilon}$ and (\ref{eqna:4.9}) we can see that $\tau_{i}>0$ for any $i=1,\cdots,p$.
Thus (\ref{eqna:4.14}) shows that
$f(\nu)+{\delta}e\prec f(\hat{x}),$ and then
it follows that $\hat{x}\notin\cX_{wE_{ap}}^{\delta}(t_0)$ by $\nu\in\cX_{wE_{ap}^{j_{t_0}}}^{\delta}$. It is a confliction to the assumption and completes the proof.
\end{proof}
\begin{remark}
(i) Combining Theorem \ref{thm:4.2} with Theorem \ref{theorem:4.8}, it shows that $\cX_{wE_{ap}}(t)\subset\cX_{wE}^{\delta}\subset\cX_{wE_{ap}}^{\varepsilon}(t)$ for any $\varepsilon, \delta$ satisfying $0<\delta<\varepsilon$ under some assumptions. Hence, we can apply two sets, i.e., $\cX_{wE_{ap}}(t)$ and $\cX_{wE_{ap}}^{\varepsilon}(t)$, to approximate the set $\cX_{wE}^{\delta}$. The form of this containment relation is similar to one of the squeeze theorem. (ii) Theorems \ref{theorem:4.8} and \ref{theorem:4.9} indicate that for given $\varepsilon$ and any $\eta$ with $0<\eta<\varepsilon$, one can yield to $\cX_{wE_{ap}}^{(\varepsilon-\eta)}(t)\subset\cX_{wE}^{\varepsilon}\subset \cX_{wE_{ap}}^{(\varepsilon+\eta)}(t)$ under some assumptions. Obviously, the result of bilateral inclusion we established more intuitively reflects the relationship of approximation.
\end{remark}

As mentioned in the previous section, Lemmas \ref{lemma:CC} and \ref{lemma:BB} show that when all possible weight vectors are taken, the set $\cX_{E_{ap}^{k_t}}(k_{t}\in[2^{t}])$ can be obtained by solving the weighted sum optimization problem for a convex multi-objective optimization problem. Obviously, it is impossible to find all the weight vectors in numerical experiments, thus (\ref{E:1}) and (\ref{2022031801}) cannot be calculated numerically, namely, Theorems \ref{thm1:4.3} and \ref{theorem:4.6}, or Theorems \ref{thm:4.2} and \ref{theorem:4.8} are only well-structured theoretical results. In order to calculate numerically, we add another constraint condition about the width of box in the design of the algorithm and prove that this algorithm can generate $\varepsilon$-efficient solutions.

{\section{Algorithm and Convergence}}
\noindent
In this section, we propose a piecewise  convexification algorithm for finding approximate efficient solutions of (MOP) and develop its convergence. Before proceeding further, let us consider  the following subproblem which need to be solved in our piecewise convexification algorithm:
\begin{align}\label{align:1-12-1}
\min\limits_{x\in Y^{k_t}} \sum\limits_{i=1}^{p} \lambda_{i}f_{i}^{k_t}(x),
\end{align}
where $\boldsymbol{\lambda}=(\lambda_{1},\cdots,\lambda_{p})\in\mathbb{R}_{++}^{p}$ is a weight vector with $\sum\limits_{j=1}^{p}\lambda_{j}=1$. This subproblem (\ref{align:1-12-1}) is referred to as the weighted sum optimization problem of the convex relaxation sub-multiobjective optimization problem $(({\rm CMOP})^{k_t})$ on $Y^{k_t}\in\mathbf{Y}^{t}$ for the subdivision $\mathbf{Y}^{t}$ of $X$. Moreover, the width of a box $Y^{k_t}=[a^{k_t}, b^{k_t}]$ is defined by $\omega(Y^{k_t}):=\|b^{k_t}-a^{k_t}\|$. Now, we turn to discuss this piecewise  convexification algorithm.
\begin{algorithm}[thb]
\captionsetup{font={scriptsize}}
\scriptsize
\caption{\scriptsize{The Piecewise Convexification Algorithm of {\rm(MOP)}}.}
\label{alg:1}
\hspace*{0.02in} {\bf Input:} 
$X=\prod\limits_{i=1}^{m}[a_i,b_i]$, $f\in\mathbb{C}^2(\mathbb{R}^n,\mathbb{R}),\varepsilon>0$, $\boldsymbol{\lambda}\in\mathbb{R}_{++}^{p}$, $\mathcal{L}_1=\emptyset$ and $\mathcal{A}=\emptyset$.\\
\hspace*{0.02in} {\bf Output:}
$\widetilde{\cX}_{E_{ap}}(t_0)$ and $\mathcal{L}_{1}.$
\begin{algorithmic}[1] 
\State Compute $\alpha^{0}:=\{\alpha_{1}^{0},\cdots,\alpha_{p}^{0}\}$ by (\ref{2022032901}) on $X^{0}$ and let $\widetilde{\alpha}=\max\{\alpha^{0}\}+0.01$. \label{line:1}
\State Estimate the minimum number of divisions $t_{\varepsilon}$ satisfying
\begin{align}\label{2022032902}
\max\limits_{k_t\in\{1,\cdots,2^{t_{\varepsilon}}\}}\dfrac{\widetilde{\alpha}}{8}\omega(Y^{k_t})^2\leq \frac{\varepsilon}{8}.
\end{align}\label{line:3}
\State Let $t_0\geq t_{\varepsilon}$ and obtain the subdivision $\mathbf{Y}^{t_0}$ of $X^0$. Let $\mathcal{L}:=\mathbf{Y}^{t_0}$. \label{line:4}

\While{$\mathcal{L}\neq \emptyset$}

\State Select a box $Y^{k_t}$ from $\mathcal{L}$ and  delete it from $\mathcal{L}$.

\State Construct $({\rm CMOP})^{k_t}$ according to {\rm(\ref{eqn:(7)})}. \label{line:5}

\State For any $j\in[p]$, let $L_{j}^{k_t}>0$ be chosen such that $L_{j}^{k_t}\geq \sqrt{m}\left|\frac{\partial}{\partial x_{i}}f_{j}\right|$ for any $i\in[p]$ and $x\in Y^{k_t}$. Let \label{line:2}
\begin{align}\label{20220406}
\widetilde{L}^{k_t}: = \min\limits_{j\in[p]}\left\{ -\frac{4L_j^{k_t}}{\tilde{\alpha}}+\sqrt{\dfrac{2\varepsilon}{\tilde{\alpha}}+\dfrac{16(L_{j}^{k_t})^{2}}{\tilde{\alpha}^2}}\right\}
\end{align}

\If{$\omega(Y^{k_t})\leq\widetilde{L}^{k_t}$}

\State Solve the subproblem {\rm(\ref{align:1-12-1})} with $\boldsymbol{\lambda}$. Let $x^{k_t, *}$ be an optimal solution and $\widetilde{\cX}_{E_{ap}^{k_t}}=\{x^{k_t,*}\}$. \label{line:6}

\State Store $Y^{k_t}$ in $\mathcal{L}_1$ and $\{x^{k_t,*}\}$ in $\mathcal{A}$.

\Else
\State Divide $Y^{k_t}$ by the direction of maximum width and obtain $Y^{k_t,1}$ and $Y^{k_t,2}$.

\State Store $Y^{k_t,1}$ and $Y^{k_t,2}$ in $\mathcal{L}$.
\EndIf
\EndWhile

\State $\widetilde{\cX}_{E_{ap}}=\left\{x\in \mathcal{A}: \nexists~y\in \mathcal{A}~\text{s.t.~}f(y)\preceq f(x)\right\}$. \label{line:7}
\end{algorithmic}
\end{algorithm}

Let us describe the settings of Algorithm \ref{alg:1} for details.
\textbf{(1)} In fact, the width of a box tends to $0$ as the number of subdivisions increases, that is, the condition $\omega(Y^{k_t})\leq\widetilde{L}^{k_t}$ must be satisfied. Thus, the While-loop is finite and  Algorithm \ref{alg:1} terminates. \textbf{(2)} In line \textbf{\ref{line:3}}, the minimum number of subdivisions $t_{\varepsilon}$ is estimated such that the width of all boxes satisfy condition (\ref{2022032902}), that is, $\max\limits_{k_t\in\{1,\cdots,2^{t_{\varepsilon}}\}}\frac{\widetilde{\alpha}}{8}\omega(Y^{k_t})^2\leq \frac{\varepsilon}{8}$. However, Theorems \ref{thm:4.2}-\ref{theorem:4.9} require the condition  $\max\limits_{k_t\in\{1,\cdots,2^{t_{\varepsilon}}\}}\frac{\widetilde{\alpha}}{8}\omega(Y^{k_t})^2\leq \varepsilon.$
\textbf{(3)} In line \textbf{\ref{line:4}}, $t_0$ is the number of divisions we set and this subdivision $\mathbf{Y}^{t_{0}}$ also satisfies condition (\ref{2022032902}).
\textbf{(4)} In line \textbf{\ref{line:5}}, we adopt the interval arithmetic method  in \cite{floudas2013deterministic,hansen2003global} to calculate the minimum eigenvalue $\lambda_{min}^{i}(x)$  and construct $({\rm CMOP})^{k_t}$.
\textbf{(5)} In  line \textbf{\ref{line:2}}, $L_j^{k_t}$ can be calculated by the interval arithmetic, see \cite{fernandez2009obtaining}. In what follows, we use this $L_{j}^{k_t}$ to bound the distance between any function value $f_{j}(x)$ on $\ Y^{k_t}$ and the optimal value of its convex relaxation function $f_{j}^{k_t}(x_{j}^{*})$ where $f_{j}^{k_t}(x_{j}^{*})=\min\limits_{x\in Y^{k_t}}f_{j}^{k_t}(x)$.  From the following theoretical analysis of this algorithm, conditions (\ref{2022032902}) and (\ref{20220406}) together guarantee that the $\varepsilon$-efficient solutions of the original problem can still be obtained when only computing a finite subset of the efficient solution set of the multi-objective convex subproblem on each sub-box. It is slightly different from the previous Theorems$\ $\ref{thm:4.2}-\ref{theorem:4.9} which only require a condition similar to (\ref{2022032902}).
\textbf{(6)} In line \textbf{\ref{line:6}}, $({\rm CMOP})^{k_t}$ is solved by the weighted sum method. According to Lemma \ref{lemma:CC}, it follows that the optimal solution $x^{k_t,*}$ of the weighted sum method with $\boldsymbol{\lambda}\in \mathbb{R}_{++}^{p}$ is an efficient solution of $({\rm CMOP})^{k_t}$, that is, $\widetilde{\cX}_{E_{ap}^{k_t}}=\{x^{k_t,*}\}\subseteq \cX_{E_{ap}^{k_t}}$ for any $k_t$.  Furthermore, this sub-box on which the weighted sum optimization problem is solved is stored in $\mathcal{L}_1$.  Obviously, when the algorithm terminates, all boxes in $\mathcal{L}_1$ just make up the subdivision of $X$.
%

Applying conditions (\ref{2022032902}) and (\ref{20220406}), we present the following analogous lemma, which was also used in \cite{niebling2019branch}. It is one of the key tools to analyze $\widetilde{\cX}_{E_{ap}}\subseteq \cX_{E}^{\varepsilon}$.
\begin{lemma}\label{lemma2022033001}
Let a box $X\in \mathbb{I}^{n}$, a twice continuously differentiable function $f:\mathbb{R}^{n}\rightarrow \mathbb{R}$, a sufficiently small positive number $\delta>0$, a constant $\alpha\geq \max\{0,-\min\limits_{x\in X}\lambda_{\min}(x)\}+\delta$, and a positive scalar $\varepsilon>0$ be given. Let $Y^{k_t}=[a^{k_t},b^{k_t}]\subseteq X$ be a sub-box and let $L^{k_t}>0$ be chosen such that $L^{k_t}>\sqrt{m}\left|\frac{\partial}{\partial x_{i}}f\right|$ for all $i\in[m]$ and $x\in Y^{k_t}$. If the width of $Y^{k_t}$ satisfies
\begin{align*}
\omega(Y^{k_t})\leq -\frac{4L^{k_t}}{\alpha} + \sqrt{\frac{(4L^{k_t})^{2}}{\alpha^{2}}+\frac{2\varepsilon}{\alpha}},
\end{align*}
then $|f(x)-\nu|\leq \dfrac{\varepsilon}{4}$ for all $x\in Y^{k_t}$ where $f^{k_t}(x):=f(x)+\dfrac{\alpha}{2}\sum\limits_{i=1}^{m}(a_{i}^{k_t}-x_{i})(b_{i}^{k_t}-x_{i})$ and $\nu:=\min\limits_{x\in Y^{k_t}}f^{k_t}(x)$. Furthermore, if $\dfrac{\alpha}{8}\omega(Y^{k_t})^2\leq \dfrac{\varepsilon}{8}$, then $|f^{k_t}(x)-\nu|\leq \dfrac{3\varepsilon}{8}$ for all $x\in Y^{k_t}$.
\end{lemma}
\begin{proof}
Using the definition of $f^{k_t}$, we can derive a bound between $f$ and $f^{k_t}$, that is,
\begin{align*}
\left|f(x)-f^{k_t}(x)\right|=\left|\frac{\alpha}{2}\sum\limits_{i=1}^{n}(a_{i}^{k_t}-x_{i})(b_{i}^{k_t}-x_{i})\right|\leq\frac{\alpha}{8}\|b^{k_t}-a^{k_t}\|^2=\frac{\alpha}{8}\omega(Y^{k_t})^2, \forall x\in Y^{k_t}.
\end{align*}
The proof of $|f(x)-\nu|\leq \frac{\varepsilon}{8}$ is identical to the proof of Lemma 2.3 in \cite{niebling2019branch} and will be omitted here. Next, we verify the second statement in the above theorem, that is, $|f^{k_t}(x)-\nu|\leq \dfrac{3\varepsilon}{8}$. Combining $|f(x)-\nu|\leq \dfrac{\varepsilon}{8}$ with the
condition $\dfrac{\alpha}{8}\omega(Y^{k_t})^2\leq \dfrac{\varepsilon}{8}$,  for all $x\in Y^{k_t}$ it is easy to see that
\begin{align*}
|f^{k_t}(x)-\nu|\leq|f^{k_t}(x)-f(x)| + |f(x)-\nu|\leq \dfrac{3\varepsilon}{8}.
\end{align*}
It is now obvious that the lemma holds.
\end{proof}
We conclude this section to briefly discuss the convergence of this algorithm, that is, Algorithm \ref{alg:1} generates a subset of all globally approximate efficient solutions of (MOP) by calculating a finite subset of the efficient solution set of the multi-objective sub-problems only.
\begin{theorem}
Let $\widetilde{\cX}_{E_{ap}}$ be the set generated by Algorithm$\ $\ref{alg:1}. Then $\ \widetilde{\cX}_{E_{ap}}\subseteq \cX_{E}^{\varepsilon}.$
\end{theorem}
\begin{proof}
If this statement would not hold, then there exists $\hat{x}\in \widetilde{\cX}_{E_{ap}}$ such that $\hat{x}\notin \cX_{E}^{\varepsilon}$. Thus, one can find $\hat{y}\in X^{0}$ satisfied $f(\hat{y})+\varepsilon e\preceq f(\hat{x})$. Obviously, $\hat{y}\notin \mathcal{A}$.  Thus we can find $Y^{k_t}\in \mathbf{Y}^{t_{0}}$ such that $\hat{y}\in Y^{k_t}\backslash\widetilde{\cX}_{E_{ap}^{k_t}}$.  This also indicates that there exists $\hat{z}\in \widetilde{\cX}_{E_{ap}^{k_t}}$ since $\widetilde{\cX}_{E_{ap}^{k_t}}\neq \emptyset$, that is, $\hat{z}$ is an optimal solution of subproblem {\rm(\ref{align:1-12-1})} with $\boldsymbol{\lambda}$. Thus, we derive
\begin{align*}
\sum\limits_{j=1}^{p}\lambda_{j}f_{j}^{k_t}(\hat{z})\leq \sum\limits_{j=1}^{p}\lambda_{j}f_{j}^{k_t}(\hat{y})
\end{align*}
Obviously, using $\boldsymbol{\lambda}=(\lambda_{1},\cdots,\lambda_{p})\in \mathbb{R}_{++}^{p}$, one can find $j_0\in[p]$ such that $f_{j_0}^{k_t}(\hat{z})\leq f_{j_0}^{k_t}(\hat{y})$. We work with the definition of $f_{j_{0}}^{k_t}$ and the condition $f(\hat{y})+\varepsilon e\preceq f(\hat{x})$, for $j_{0}$ it leads to the chain of inequalities
\begin{align*}
f_{j_0}(\hat{z})+\frac{\alpha_{j_0}^{k_t}}{2}\sum\limits_{i=1}^{m}(a_{i}^{k_t}-\hat{z}_{i})(b_{i}^{k_t}-\hat{z}_{i})\leq f_{j_{0}}(\hat{x})-\varepsilon.
\end{align*}
Due to the first condition of $(\ref{2022032902})$ and  $\alpha_{j_0}^{k_t}\leq \widetilde{\alpha}$, it follows that
\begin{align}\label{2022033003}
0\leq\frac{\alpha_{j_0}^{k_t}}{2}\sum\limits_{i=1}^{m}(a_{i}^{k_t}-\hat{z}_{i})(b_{i}^{k_t}-\hat{z}_{i})+\frac{\varepsilon}{8}<\frac{\alpha_{j_0}^{k_t}}{2}\sum\limits_{i=1}^{m}(a_{i}^{k_t}-\hat{z}_{i})(b_{i}^{k_t}-\hat{z}_{i})+\varepsilon,
\end{align}
which leads to $f_{j_0}(\hat{z})<f_{j_0}(\hat{x})$.

For any $j\in\{1,\cdots,p\}\setminus \{j_0\}$, based on the second condition of (\ref{2022032902}) and lemma \ref{lemma2022033001}, we can show that
\begin{align*}
|f_{j}^{k_t}(\hat{y})-f_{j}^{k_t}(\hat{z})|\leq |f_{j}^{k_t}(\hat{y})-\nu_{j}^{k_t}| + |f_{j}^{k_t}(\hat{z})-\nu_{j}^{k_t}|\leq \frac{3\varepsilon}{4}
\end{align*}
where $\nu_{j}^{k_t}=\min\limits_{x\in Y^{k_t}}f_{j}^{k_t}(x)$. The above inequality implies that $f_{j}^{k_t}(\hat{z})-\frac{3\varepsilon}{4}\leq f_{j}^{k_t}(\hat{y})$. Using the assumption condition $f(\hat{y})+\varepsilon e\preceq f(\hat{x})$, and the definition of $f_{j}^{k_t}$,  we also obtain
\begin{align}\label{2022033002}
f_{j}(\hat{z})+\frac{\alpha_{j}^{k_t}}{2}\sum\limits_{i=1}^{m}(a_{i}^{k_t}-\hat{z}_i)(b_{i}^{k_t}-\hat{z}_{i})-\frac{3\varepsilon}{4}\leq f(\hat{x})-\varepsilon,
\end{align}
Similar to (\ref{2022033003}), the inequality (\ref{2022033002}) yields to $f_{j}(\hat{z})\leq f_{j}(\hat{x})$ for any $j\in\{1,\dots,p\}\setminus\{j_0\}$. Together with $f_{j_0}(\hat{z})<f_{j_0}(\hat{x})$, we now obtain $f(\hat{z})\preceq f(\hat{x})$ with $\hat{z}\in\widetilde{\cX}_{E_{ap}^{k_t}}\subseteq \mathcal{A}$. This contradicts the assumption that $\hat{x}\in\widetilde{\cX}_{E_{ap}}$, that is, the statement $\widetilde{\cX}_{E_{ap}}\subseteq \cX_{E}^{\varepsilon}$ holds.
\end{proof}

%

\section{Numerical Experiments}
In this section, we demonstrate  the numerical performance
of our Algorithm \ref{alg:1} by some examples, which include the nonconvex and  disjointed Pareto front, respectively. All computations have been performed on a computer with Inter(R) Core(TM)i5-8250U CPU and 8 Gbytes RAM. For all instances, we take $\varepsilon=0.02$ and the weight vector $\boldsymbol{\lambda}=(\frac{1}{p},\cdots,\frac{1}{p})$ where $p$ is the number of objective functions.
\begin{example}\label{example:20220317}
This test instance was proposed in {\cite{niebling2019branch}}:
\begin{equation*}       
f(x)=\left(                 
  \begin{array}{c}   
    x_{1} \\  
    \\
    \dfrac{1}{x_{1}}\Big(2-\exp\Big(-\Big(\frac{x_{2}-0.2}{0.004}\Big)^2\Big)-0.8\exp\Big(-\Big(\frac{x_{2}-0.6}{0.4}\Big)^2\Big)\Big) \\  
  \end{array}
\right)               
\end{equation*}
with $X=[0.1, 1]\times [0,1]$. As stated in \cite{niebling2019branch}, the globally efficient solutions are $(\tilde{x}_{1}, \tilde{x}_{2})$ with $\tilde{x}_{2}\approx 0.2$ and $\tilde{x}_{1}\in [0.1, 1]$. This example also exists locally efficient solutions with $\tilde{x}_{2}\approx 0.6$ and $\tilde{x}_{1}\in [0.1, 1]$. The graph of $f_2$ and the image space are showed in Figure \ref{example1-image:2022031701}, respectively.
\end{example}
\begin{figure}[H]
\captionsetup{font={scriptsize}}
\subfigure[\scriptsize{The graph of $f_{2}(x)$}]
{
\begin{minipage}[t]{0.5\linewidth}
\includegraphics[width=6.5cm]{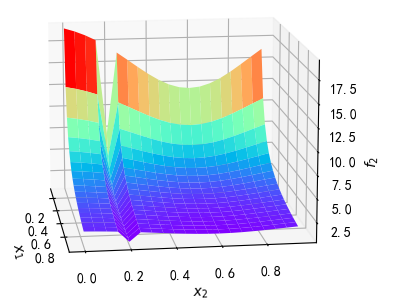}
\end{minipage}%
}%
\subfigure[\scriptsize{Image space}]{
\begin{minipage}[t]{0.5\linewidth}
\includegraphics[width=7cm]{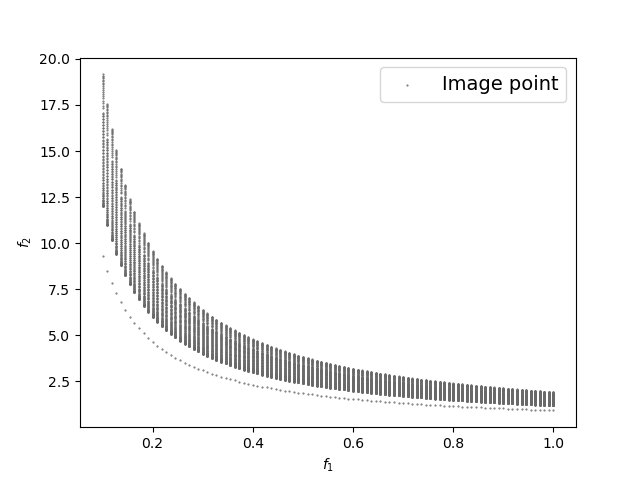}
\end{minipage}%
}%
\centering
\caption{Example \ref{example:20220317}}\label{example1-image:2022031701}
\end{figure}
In this example, we roughly estimate ${t}_{\varepsilon}=12$ and take $t_{0}=12$. Figure \ref{example1-solution:2022031702} shows the $\varepsilon$-efficient solutions and the non-dominated points of this example.
\begin{figure}[H]
\captionsetup{font={scriptsize}}
\subfigure[\scriptsize{$\varepsilon$-efficient solutions}]
{
\begin{minipage}[t]{0.5\linewidth}
\includegraphics[width=7cm]{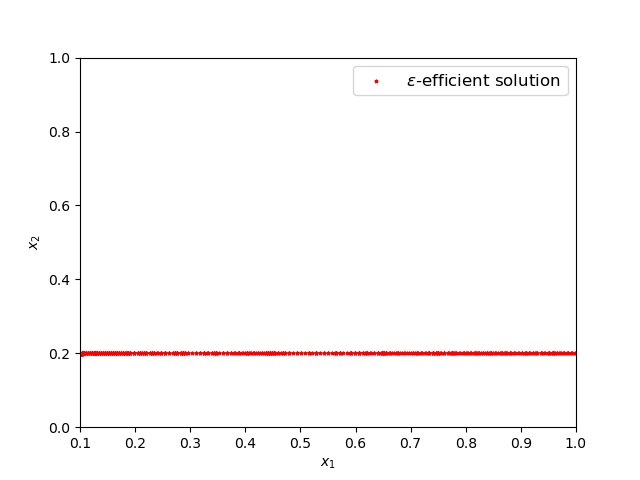}
\end{minipage}%
}%
\subfigure[\scriptsize{$\varepsilon$-nondominated points}]{
\begin{minipage}[t]{0.5\linewidth}
\includegraphics[width=7cm]{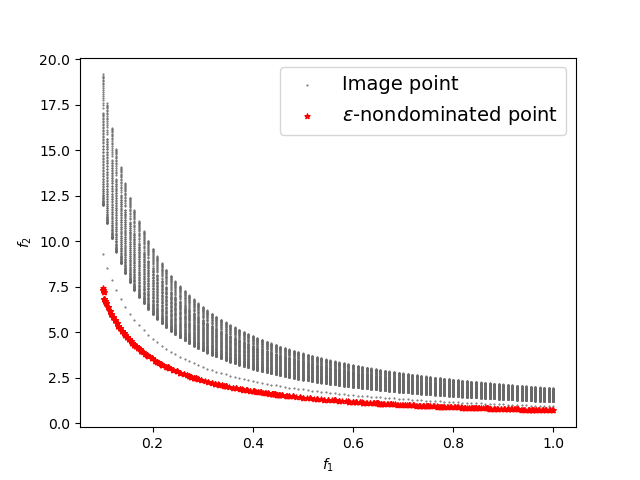}
\end{minipage}%
}%
\centering
\caption{$\varepsilon$-efficient solutions and the non-dominated points of Example \ref{example:20220317}}\label{example1-solution:2022031702}
\end{figure}
Figure \ref{example1-solution:2022031702} (a) and (b) show performance in terms of how well the globally efficient solutions  are found and the Pareto front is approximated, respectively,  by Algorithm \ref{alg:1}. In other words, the points generated by Algorithm \ref{alg:1} are globally approximate solutions, not locally ones.

\begin{example}\label{example:2022031702}
This test instance is based on {\cite{niebling2019branch}}:
\begin{equation*}       
f(x)=\left(                 
  \begin{array}{c}   
    1 - \exp\Big(-\sum\limits_{i=1}^{3}\Big(x_{i}-\frac{1}{\sqrt{3}}\Big)^2\Big)  \\  
    \\
    1 - \exp\Big(-\sum\limits_{i=1}^{3}\Big(x_{i}+\frac{1}{\sqrt{3}}\Big)^2\Big) \\  
  \end{array}
\right) 
with ~
X=\begin{bmatrix}
\left(                 
  \begin{array}{c}   
    -2\\  
    -2\\
    -2\\  
  \end{array}
\right),
&
\left(                 
  \begin{array}{c}   
    2\\  
    2\\
    2\\  
  \end{array}
\right)
\end{bmatrix}.
\end{equation*}
\end{example}
Similarly, by roughly estimating $t_{\varepsilon}=16$. Thus, let $t_{0}=16$. The results of $\varepsilon$-efficient solutions and non-dominated points  as illustrated in Figure \ref{example2-results:1-0}, respectively. Obviously, the shape of the approximate Pareto front can be characterized by this piecewise convexification algorithm from it.
\begin{figure}[h]
\captionsetup{font={scriptsize}}
\subfigtopskip=2pt 
\subfigbottomskip=0.0cm 
\subfigcapskip=0pt 
\subfigure[\scriptsize{$\varepsilon$-efficient solutions}]
{
\begin{minipage}[t]{0.5\linewidth}
\includegraphics[width=6.5cm]{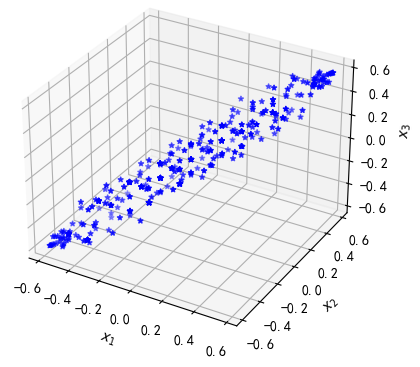}
\end{minipage}%
}%
\subfigure[\scriptsize{$\varepsilon$-nondominated points}]{
\begin{minipage}[t]{0.5\linewidth}
\includegraphics[width=7cm]{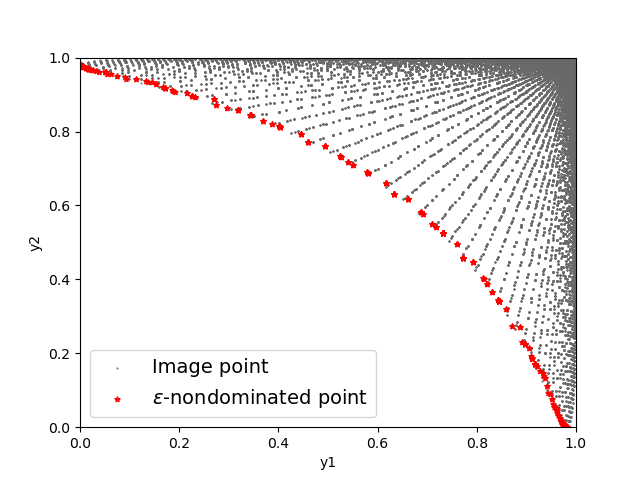}
\end{minipage}%
}%
\centering
\caption{The $\varepsilon$-efficient solutions and non-dominated points of Example \ref{example:2022031702}}\label{example2-results:1-0}\vspace*{-0.4cm}
\end{figure}

\begin{example}\label{example:4}
This example is considered in {\cite{zitzler2000comparison}}:
\begin{equation*}       
f(x)=\left(                 
  \begin{array}{c}   
    x_{1} \\  
    \\
    1+9x_{2}-(x_{1}(1+9x_{2}))^{\frac{1}{2}}-x_{1}\sin(10\pi x_{1}) \\  
  \end{array}
\right)
with ~
X=\begin{bmatrix}
\left(                 
  \begin{array}{c}   
    0.1\\  
    0\\
  \end{array}
\right),
&
\left(                 
  \begin{array}{c}   
    1\\  
    1\\
  \end{array}
\right)
\end{bmatrix}.               
\end{equation*}
The $f_{2}=1-\sqrt{f_{1}}-f_{1}\sin(10\pi f_{1})$ that consists of several noncontiguous convex parts is Pareto front. Figure \ref{example3:3-0} shows the graph of $f_2$ and the image space, respectively.\end{example}
\begin{figure}[H]
\captionsetup{font={scriptsize}}
\subfigure[\scriptsize{The graph of $f_{2}(x)$}]
{
\begin{minipage}[t]{0.5\linewidth}\hspace*{-0.4cm}
\includegraphics[width=7cm]{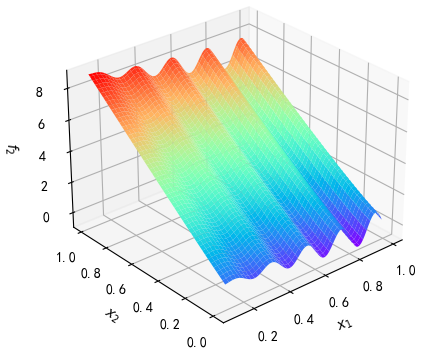}
\end{minipage}%
}%
\subfigure[\scriptsize{Image space}]{
\begin{minipage}[t]{0.5\linewidth}
\includegraphics[width=7cm]{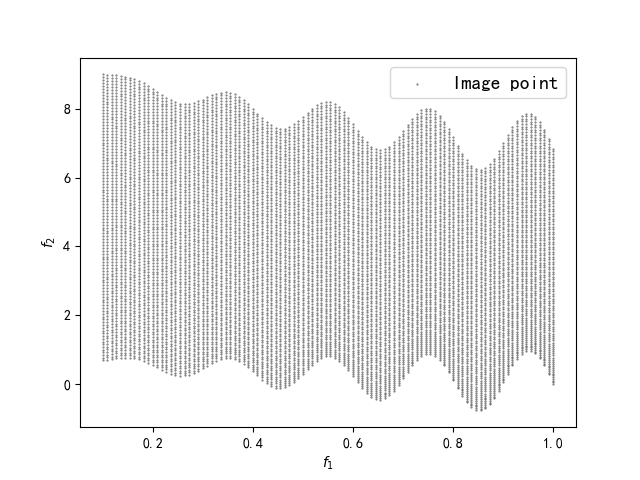}
\end{minipage}%
}%
\centering
\caption{Example \ref{example:4}}\label{example3:3-0}
\end{figure}
\begin{figure}[h]
\captionsetup{font={scriptsize}}
\centering
\subfigure[$t_0$ = 13]{
\begin{minipage}[t]{0.5\linewidth}
\includegraphics[width=6.4cm]{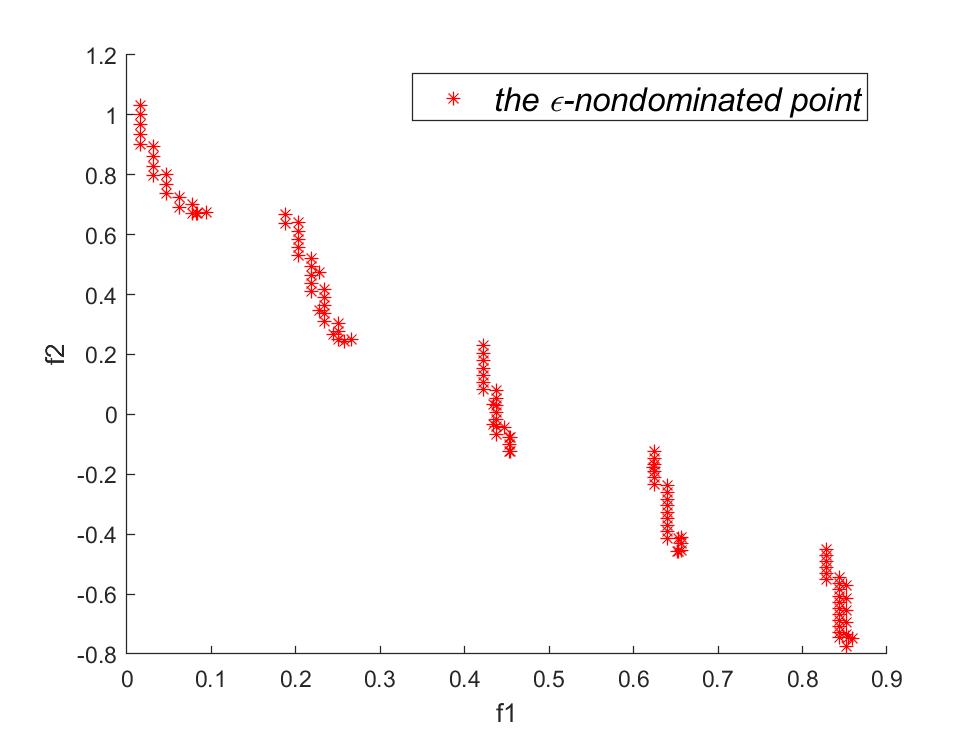}
\end{minipage}%
}%
\subfigure[$t_0$ = 15]{
\begin{minipage}[t]{0.5\linewidth}
\includegraphics[width=6.4cm]{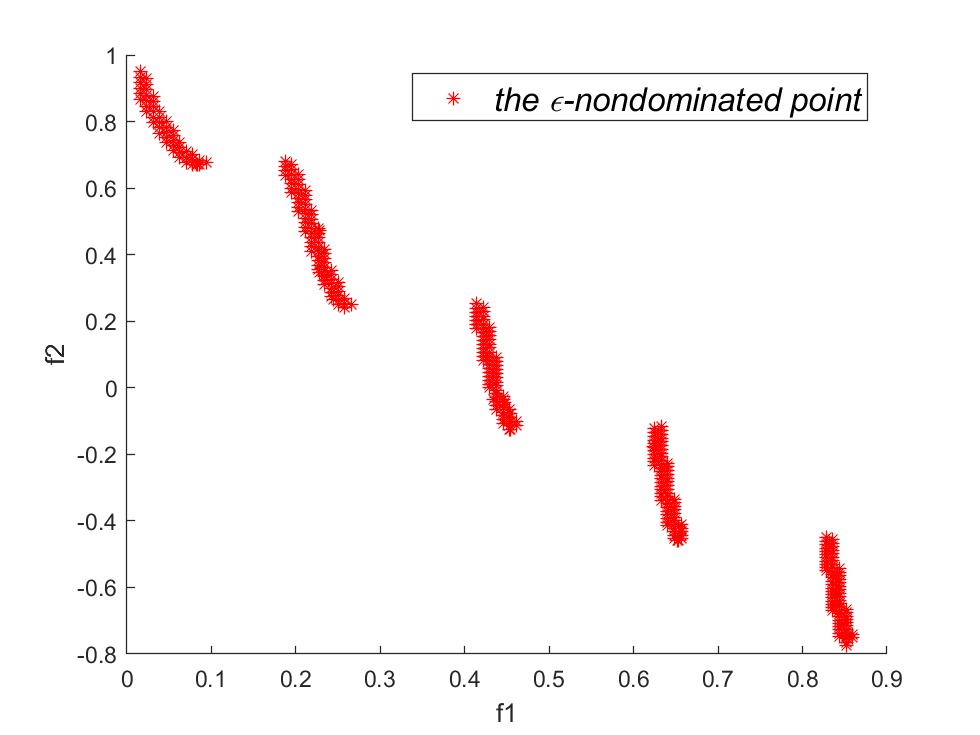}
\end{minipage}
}%
\caption{The solutions of the convexification method with different $t_0$}\label{fig:3-1}
\end{figure}
In the same way, we roughly estimate $t_{\varepsilon}= 12$ and analysis the influences of $t_0$, which are shown
in Figure \ref{fig:3-1}. It is apparent from Figure \ref{fig:3-1} that these non-dominated points are almost uniformly distributed with increasing the subdivision times, that is, $t_{0}=13$ to $t_{0}=15$. More importantly, Figure \ref{fig:3-1} indicates that Algorithm \ref{alg:1} is also suitable for deal with this complex multi-objective optimization problems.

Further, we verify that when the number of objectives is four, the piecewise convexification method also has practical operability.
\begin{example}\label{2022032050}
A multi-objective rocket injector design problem was studied in \cite{burachik2017new} as follows:
\begin{align*}
\min~[f_{1}, f_{2}, f_{3}, f_{4}]~~ \text{subject~to~} 0\leq x_{1}, x_{2}, x_{3},x_{4}\leq 1
\end{align*}
where
\begin{align*}
f_{1} =~ &0.692 + 0.477x_{1}- 0.678x_{4}-0.08x_{3}-0.065x_{2}-0.167x_{1}^{2} - 0.0129x_{1}x_{4}\\
        &+ 0.0796x_{4}^{2} - 0.0634x_{1}x_{3} - 0.0257x_{3}x_{4} + 0.0877x_{3}^{2} - 0.0521x_{1}x_{2}\\
        &+ 0.00156x_{2}x_{4} + 0.00198x_{2}x_{3} + 0.0184x_{2}^{2},\\
f_{2} =~ & 0.37 - 0.205x_{1} + 0.0307x_{4} + 0.108x_{3} + 1.019x_{2}-0.135x_{1}^{2} + 0.0141x_{1}x_{4}\\
        & + 0.0998x_{4}^{2} + 0.208x_{1}x_{3} - 0.0301x_{3}x_{4} - 0.226x_{3}^{2} + 0.353x_{1}x_{2} - 0.0497x_{2}x_{3}\\
        & - 0.423x_{2}^{2} + 0.202x_{1}^{2}x_{4} - 0.281x_{1}^{2}x_{3} - 0.342x_{1}x_{4}^{2} -0.245x_{3}x_{4}^{2} + 0.281x_{3}^{2}x_{4}\\
        & -0.184x_{1}x_{2}^{2} + 0.281x_{1}x_{3}x_{4},\\
f_{3} = ~& 0.153 - 0.322x_{1} + 0.396x_{4} + 0.424x_{3} + 0.0226x_{2} + 0.175x_{1}^{2} + 0.0185x_{1}x_{4}\\
        & - 0.0701x_{4}^{2} - 0.251x_{1}x_{3} + 0.179x_{3}x_{4} + 0.015x_{3}^{2} + 0.0134x_{1}x_{2} + 0.0296x_{2}x_{4}\\
        & + 0.0752x_{2}x_{3} + 0.0192x_{2}^{2},\\
f_{4} = ~&0.758 + 0.358x_{1} - 0.807x_{4} + 0.0925x_{3} - 0.0468x_{2} - 0.172x_{1}^{2} + 0.0106x_{1}x_{4}\\
         & + 0.0697x_{4}^{2} - 0.146x_{1}x_{3} - 0.0416x_{3}x_{4} + 0.102x_{3}^{2} - 0.0694x_{1}x_{2}\\
         & - 0.00503x_{2}x_{4} + 0.0151x_{2}x_{3} + 0.0173x_{2}^{2}.
\end{align*}
\end{example}
Similarly, $t_{0}=12$ satisfies the condition, then taking $t_{0}=13$. Analogously to \cite{burachik2017new}, we display the projection $(f_{1},f_{2},f_{3},f_{4})$ in $f_{1}f_{2}f_{3}$-space and $f_{2}f_{3}f_{4}$-space in  Figure \ref{fig:2022032001}.
\begin{figure}[H]
\captionsetup{font={scriptsize}}
\subfigure[\scriptsize{The Pareto front to the $f_1f_{2}f_{3}$-space}]
{
\begin{minipage}[t]{0.5\linewidth}\hspace*{-0.3cm}
\includegraphics[width=6.5cm]{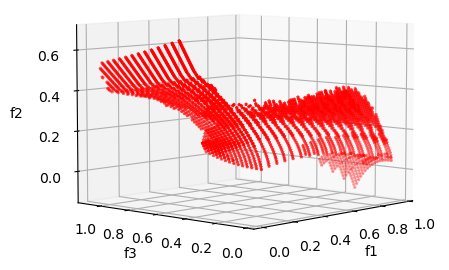}
\end{minipage}%
}%
\subfigure[\scriptsize{The Pareto front to the $f_2f_{3}f_{4}$-space}]{
\begin{minipage}[t]{0.5\linewidth}
\includegraphics[width=6.0cm]{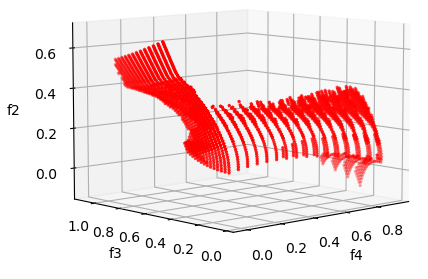}
\end{minipage}%
}%
\caption{Projected Pareto fronts for Example \ref{2022032050}}\label{fig:2022032001}
\end{figure}

\section{Conclusions}\label{secremark}
In this paper, we employ the  ${\rm \alpha BB}$ method and the interval division technique to obtain a series of convex relaxation sub-multiobjective problems, which is regarded as the piecewise convexification problem of MOP on the whole set $X$.  Furthermore, we establish the (approximate, weakly) efficient solution set of this piecewise convexification problem and investigate that the globally approximate (weakly) efficient solution set of MOP can be approximated by combining with two (approximate) solution sets of the piecewise convexification problem, that is, $\cX_{wE_{ap}}^{(\varepsilon-\eta)}(t)\subset\cX_{wE}^{\varepsilon}\subset \cX_{wE_{ap}}^{(\varepsilon+\eta)}(t)$ or $\cX_{E_{ap}}^{(\varepsilon-\eta)}(t)\subset\cX_{E}^{\varepsilon}\subset \cX_{E_{ap}}^{(\varepsilon+\eta)}(t)$. Although these sets cannot be calculated numerically, the theoretical
results are still nice and well-structured. In order to calculate, we apply two condition about the width of a box to design a piecewise convexification algorithm. This algorithm also yields a subset of the globally approximated efficient solution set of MOP by only computing a finite subset of the efficient solution set of multi-objective subproblem on every sub-box.

Different from the traditional $\rm{\alpha BB}$ method, we pay more attention to establishing the theoretical properties of the piecewise convexification method itself in this paper, rather than emphasizing the design of the algorithm.
Therefore, in numerical experiments, we only verified the feasibility of the piecewise convexification algorithm and did not demonstrate its effectiveness. In fact, this method with well theoretical properties deserves our further study and this article can be seen as the first of a series of papers about this piecewise convexification method for us.
In our future research work,  we will redefine the solution set of the piecewise convexification problem, which takes advantage of the  previous division information. At the same time, the evolutionary algorithm and a new  subdivision strategy will be incorporated into the piecewise convexification method to quickly delete boxes and speed up the algorithm. Moreover, a new termination condition will be added to the algorithm to avoid calculating the minimum number of divisions. These researches are underway and some interesting results would be obtained.

\hbox to14cm{\hrulefill}\par
\noindent\textsc{Qiao Zhu}\\
{{\footnotesize
{College of Mathematics, Sichuan University, 610065, Chengdu Sichuan, China}\\
 {E-mail address: math\_qiaozhu@163.com}\\
{\scshape Liping Tang}\\
{\footnotesize
{National Center for Applied Mathematics, Chongqing Normal University, 401331 Chongqing, China.}\\
 { Email address: tanglipings@163.com}}
\medskip}\\
{\scshape Xinmin Yang}\\
{\footnotesize
{National Center for Applied Mathematics, Chongqing Normal University, 401331 Chongqing, China. }\\
 { Email: xmyang@cqnu.edu.cn}}\\

\end{document}